\documentclass[preprint,12pt]{elsarticle}
\usepackage{amssymb,amsthm,amsmath}
\RequirePackage[colorlinks,citecolor=blue,urlcolor=blue]{hyperref}

\theoremstyle{plain}
\newtheorem{theorem}{Theorem}[section]

\newtheorem{lemma}{Lemma}[section]

\newtheorem{example}{Example}[section]
\newtheorem{corollary}{Corollary}[section]

\newcommand{\R}{\mathbb R}

\begin{document}

\begin{frontmatter}
 
\title{Rates of convergence towards the Fr\'echet distribution}

\author[ca]{Carine Bartholm\'e}
\ead{cabartho@ulb.ac.be}
\author[yv]{Yvik Swan}
\ead{yvik.swan@uni.lu}

 \address[ca]{Universit\'e libre de Bruxelles}
 \address[yv]{Universit\'e du Luxembourg}

\begin{abstract}
 We develop Stein's method for the Fr\'echet
  distribution and apply it to compute rates of convergence in
  distribution of renormalized sample maxima to the Fr\'echet
  distribution. 
\end{abstract}


\end{frontmatter}

\section{Introduction}
Let $X_1, X_2, \ldots$ be independent random variables with common
distribution function $F$ and let $M_n = \max(X_1, X_2, \ldots,
X_n)$. The distribution $F$ is in the domain of attraction of a
Fr\'echet distribution with index  $\alpha>0$, and we write $F \in DA(\alpha)$, if
there exist normalizing constants $a_n>0$ and $b_n$ such that
\begin{equation}
  \label{eq:1}
  P \left[ (M_n-b_n)/a_n \le x\right] = F(a_nx+b_n)^n
  \longrightarrow \Phi_{\alpha}(x) \mbox{ as } n \to \infty,
\end{equation}
where $\Phi_{\alpha}(x) = \mathrm{exp}(-x^{-\alpha})\mathbb{I}(x\ge0)$
is the Fr\'echet cumulative distribution. 

 Specific sufficient and
necessary conditions on $F$ for \eqref{eq:1} to hold were provided by
Gnedenko \cite{gnedenko1943distribution} in 1943; the following
reformulation of Gnedenko's result is taken from  Smith
  \cite{smith1982uniform}. 

\begin{theorem}[Gnedenko \cite{gnedenko1943distribution}, Smith
  \cite{smith1982uniform}]\label{thegnedenko1943}
Let $L(t) = -t^{\alpha}\log F(t)$ for $t>0$. Then $F \in DA(\alpha)$ if,
and only if,   $F(x)<1$ for all
$x<\infty$ and  $L(t)$ is
slowly varying at $\infty$.
\end{theorem}

There exist no Berry-Esseen type results for Fr\'echet convergence
because rates can vary much depending on the properties of
$F$. However, good control on the function $L(t)$ allows to determine
a sequence $a_n$ for which precise rates of convergence in
\eqref{eq:1} are readily obtained; more precisely
\cite{smith1982uniform,omey1988rate} prove that
\begin{equation}
  \label{eq:6}
\sup_x \left| F(a_nx)^n - \Phi_{\alpha}(x)  \right| = O(r_n)
\end{equation}
where both $a_n$ and $r_n$ are explicit quantities depending
explicitly on $L$, see forthcoming Theorem \ref{theo:1} for details.

In this note we use a new version of the so-called \emph{Stein's method} (see, e.g.,
\cite{NP11,BC05,ChGoSh11} for an overview) to provide explicit (fixed $n$) bounds on
$D(M_n, G)$ between the law of $M_n$ and the Fr\'echet, with $D(\cdot,
\cdot)$ a probability distance.  We stress the fact that our take on the Stein's method rests
on new identities for the Fr\'echet which do not fit within the
recently developed general approaches to  Stein's method  via diffusions 
\cite{EdVi12,KuTu11} or via the so-called density approach
\cite{Stein2004,LS12a}.

\section{Stein's method for the Fr\'echet   distribution}
\label{sec:frech-extr-value}

Fix throughout $\alpha>0$ and write $G \sim \Phi_{\alpha}$.  
Let $\mathcal{F}(\alpha)$ be the collection of all differentiable functions on
$\R$ such that $$\lim_{x \to \infty}\varphi(x)  e^{-x^{-\alpha}}=\lim_{x
  \to 0}\varphi(x)  e^{-x^{-\alpha}} =0.$$ Note that the second
condition is satisfied as soon as $\varphi$ is well-behaved at $0$,
whereas the first condition requires that we 
impose $\lim_{x \to \infty}\varphi(x) = 0$.  
We define the differential
operator  
\begin{equation}
  \label{eq:38}
  \mathcal{T}_{\alpha}\varphi(x) = \varphi'(x) x^{\alpha+1} + \alpha
  \varphi(x)  
\end{equation}
 for all $\varphi\in \mathcal{F}(\alpha)$. Direct computations show
that operator \eqref{eq:38} satisfies the integration by parts formula 
 \begin{equation}
   \label{eq:12}
   E \left[f(G) \mathcal{T}_{\alpha}\varphi(G)  \right] =- E \left[
     G^{\alpha+1}f'(G) \varphi(G) \right]
 \end{equation}
for all $\varphi \in \mathcal{F}(\alpha)$ and all $f$ such that
$|f(0)|<\infty$ and $\lim_{x \to \infty} |f(x)|<\infty$.  In
particular $  E \left[ \mathcal{T}_{\alpha}\varphi(G)
\right] = 0 $
 for all $\varphi\in \mathcal{F}(\alpha)$.

Following the custom in Stein's method we consider  Stein equations of
the form 
\begin{equation}
  \label{eq:4}
  h(x) - E h(G) = \mathcal{T}_{\alpha }\varphi(x)
\end{equation}
for   $h(x)$  some function such that $E \left| h(G)
 \right|< \infty$.
We write the solution of  \eqref{eq:4} as  $\varphi_{h} := 
\mathcal{T}_{\alpha}^{-1}h$.  By definition of $\mathcal{T}_{\alpha}$
these are given by  
\begin{align}\label{eq:solstfre}
  \varphi_h(x) =  e^{x^{-\alpha}} \int_0^{x} \left( h(y)-E \left[
      h(G) \right] \right)  y^{-\alpha-1}e^{-y^{-\alpha}}dy
\end{align}
or, equivalently, 
\begin{align}
  \label{eq:solstfre2}
    \varphi_h(x) =  e^{x^{-\alpha}}  \int_{x}^{\infty} \left(E \left[
      h(G) \right] - h(y) \right) y^{-\alpha-1} e^{-y^{-\alpha}}dy 
\end{align}
for all $x \ge0$. 
 By l'Hospital's rule,  this function satisfies 
$  \lim_{x \to 0 } \varphi_h(x) 
=  \lim_{x \to 0} \left( h(x)-E \left[
      h(G) \right] \right)= (h(0) - E h(G))/\alpha.$
Also we use the fact that $E |h(G)|<\infty$ to deduce  that $\lim_{x \to \infty} \varphi_h(x) = 0$. 
Hence, in particular, this $\varphi_h$  belong to
$\mathcal{F}(\alpha)$ under reasonable assumptions on the function
$h$.  
More precise estimates will be given in the sequel, as the
need arises.

Next take $W_n = (M_n-b_n)/a_n$ with $a_n, b_n$ and $M_n$ as in the
Introduction. Let $\mathcal{S} = \left\{ x : F'(x)>0 \right\}$ and 
suppose that $\mathcal{S}$ in an interval  with closure
$\bar{\mathcal{S}} = [a, b]$. Let $f_n$ be the probability density function
of $W_n$; further suppose that $f_n$ is differentiable on $I$ except
perhaps in a finite number of points and define 
  \begin{equation}
    \label{eq:30}
    \rho_n(x)  = (\log f_n(x))'= a_n \left( (n-1)\frac{f(a_nx+b_n)}{F(a_nx+b_n)} + 
    \frac{f'(a_nx+b_n)}{f(a_nx+b_n)} \right)
  \end{equation}
 (the score function of $W_n$).

Let $ \mathcal{F}(f_n)$ be the collection of all  differentiable
functions on $\R$ such that $\lim_{x \to a^{+}} \varphi(x)x^{\alpha+1} f_n(x) = \lim_{x \to b^{-}}
\varphi(x) x^{\alpha+1}f_n(x) =0$ and  define the  differential operator 
\begin{align}
  \mathcal{T}_n\varphi(x) & = \varphi'(x) x^{\alpha+1}  + \varphi(x)
  x^{\alpha}  \left(  \alpha+1   + x \rho_n(x)\right),  \label{eq:23}
\end{align}
 for all $\varphi
\in \mathcal{F}(f_n)$. Then straightforward computations show that
\eqref{eq:23} satisfies 
\begin{equation}
  \label{eq:24}
  E  \left[ \mathcal{T}_{n}\varphi(W_n)  \right] = 0
\end{equation}
for all $\varphi \in \mathcal{F}(f_n)$.

Our next result is an
immediate consequence of the above identities.
\begin{lemma}\label{lem:steins-meth-frech}
Let $W_n$ and $G$ be as above, and let $\varphi_h =
\mathcal{T}_{\alpha}^{-1}h$. For all $h$ such that  $\varphi_h \in
\mathcal{F}(f_n)$ we have
\begin{equation} \label{eq:31} 
  E h(W_n) - E h(G)
   = \alpha E \left[ \varphi_h(W_n)  \left(1 -  W_n^{\alpha} \left(
        \frac{\alpha+1}{\alpha} +  \frac{W_n}{\alpha}\rho_n(W_n)
      \right) \right) \right]
\end{equation}
for all $h$.
\end{lemma}
\begin{proof}
Taking $\varphi$ solution of \eqref{eq:4}, we
use \eqref{eq:24} to deduce
  \begin{align*}
     E h(W_n) - E h(G) & = E \left[ \mathcal{T}_{\alpha}\varphi(W_{n})
     \right] = E \left[ \mathcal{T}_{\alpha}\varphi(W_{n})
     \right] - E \left[ \mathcal{T}_n\varphi(W_n)
     \right].
  \end{align*}
Conclusion follows by definition of $\mathcal{T}_{\alpha}$ and $\mathcal{T}_n$. 
\end{proof}

\begin{example}\label{ex:pareto}
The standard example of convergence towards the Fr\'echet is the
maximum of Pareto random variables. Take  $F(x) =
(1-x^{-\alpha})\mathbb{I}(x \ge 1)$ and fix  $a_n=n^{1/\alpha}$,
$b_n=0$. The cdf and pdf of $W_n = M_n/n^{1/\alpha}$ have support $[n^{-1/\alpha}, +\infty)$ and are
\begin{equation}
  \label{eq:2}
F_n(x) =  \left( 1-\frac{x^{-\alpha}}{n} \right)^{n} \mbox{ and }  f_n(x) = \alpha x^{-\alpha-1} \left( 1-\frac{x^{-\alpha}}{n} \right)^{n-1},
\end{equation}
respectively. The score function is 
\begin{equation*}
   \rho_n(x)= -\frac{\alpha+1}{x} +
  \frac{n-1}{n}\frac{\alpha}{x^{\alpha+1}} \left( 1-\frac{x^{-\alpha}}{n} \right)^{-1}
\end{equation*}
so that, from \eqref{eq:31}, 
\begin{equation*}
E h(W_n) - E h(G)
   = \alpha E \left[ \varphi_h(W_n) \left( 1-\frac{n-1}{n} \left(
         1-\frac{W_n^{-\alpha}}{n} \right)^{-1}  \right) \right]
\end{equation*}
with $\varphi_h = \mathcal{T}_{\alpha}^{-1}h$. We readily compute 
\begin{equation*}
  E \left[ \left| 1-\frac{n-1}{n} \left( 1-\frac{W_n^{-\alpha}}{n}
      \right)^{-1}  \right|  \right] = \frac{2}{n-1}\left( 1-\frac{1}{n} \right)^n.
\end{equation*}
Finally, for 
$h(u)= \mathbb{I}(u \le x)$, the solutions $\varphi_h(u)$ satisfy $\|
\varphi_h\|_{\infty} \le 1/\alpha$ (this is proved in the forthcoming
Lemma \ref{lem:indic}). Hence 
\begin{align*}
  \sup_{x \in \R} |P(W_n \le x) - P(G \le x)| \le  \frac{2e^{-1}}{n-1}.
\end{align*}
\end{example}

\section{Uniform rates of convergence towards Fr\'echet distribution}
\label{sec:rates-conv-towards}

In light of \eqref{eq:31} as well forthcoming Lemma \ref{lem:indic}
we immediately deduce 
$ \sup_{x \in \R} \left| F^n(a_nx) - \Phi_{\alpha}(x)\right| \le \Delta_n$
with
\begin{align*}
\Delta_n=E \left|   1 -  W_n^{\alpha} \left(
        \frac{\alpha+1}{\alpha} +  \frac{W_n}{\alpha}\rho_n(W_n)
      \right) \right|.
\end{align*}
Hence, if 
\begin{align*}
\rho_n(x) = \frac{\alpha}{x^{\alpha+1}} O(1+r_n)-  \frac{\alpha+1}{x}O(1)
\end{align*}
then we have convergence towards the Fr\'echet.  This last condition
is equivalent to 
\begin{align*}
  f_n(x) = x^{-\alpha-1}e^{-x^{-\alpha}}O(1+r_n) = \phi_{\alpha}(x) O(1+r_n).
\end{align*}
with $r_n\to0$, and thus a very transparent explanation of why
$\Delta_n\to 0$ implies convergence towards the Fr\'echet.

 More generally, suppose as in 
 \cite{smith1982uniform, omey1988rate} that  $L(t)$ as defined in
 Theorem \ref{thegnedenko1943}  is slowly varying with remainder $g$, that is
  take 
  \begin{equation}
    \label{eq:3}
\lim_{t\to \infty}    \frac{L(tx)}{L(t)} -1 = O(g(t)). 
  \end{equation}
Let $b_n=0$ and $a_n$ be such that 
$   -\log F(a_n) \le n^{-1} \le -\log F(a_n^{-})$. Then \cite[equation
(2.2)]{smith1982uniform}
\begin{align*}
  -n \log F(a_n) = 1+O(g(a_n))
\end{align*}
so that, as in  \cite[page 602]{smith1982uniform},
\begin{align}
  -n\log F(a_nx) & = - n \frac{\log F(a_nx)}{\log F(a_n)} \log F(a_n)
  \nonumber\\
  & = x^{-\alpha}\frac{L(a_nx)}{L(a_n)}  \left( 1+O(g(a_n)) \right) \nonumber\\
  & = x^{-\alpha} \left( 1+O(g(a_n)) \right) \label{eq:5} 
\end{align}
and thus 
\begin{equation*}
  F^n(a_nx) - \Phi_{\alpha}(x) = O(r_n)
\end{equation*}
with $r_n = g(a_n)$. Hence the question of pointwise convergence of
the law of $M_n/a_n$ towards $G$ is settled. 

There remains the problem
of determining rates at which the convergence takes place; this
problem is tackled and solved in \cite{smith1982uniform} under
further assumptions on $L$.
Here  we simply apply  Lemma~\ref{lem:steins-meth-frech}.  
\begin{theorem} \label{theo:1} 
Suppose that $F$  satisfies \eqref{eq:3}. Let $b_n=0$ and $a_n$ be such that 
$   -\log F(a_n) \le n^{-1} \le -\log F(a_n^{-})$ and define $W_n = M_n/a_n$. Then 
\begin{equation}
\label{eq:13} 
  E h(W_n) - E h(G) =- \alpha E \left[
    \mathcal{T}_{\alpha}^{-1}h(W_n) \right]O(g(a_n))
\end{equation}
for all $h$ such that $E |h(W_n)|<\infty$ and $E |h(G)|<\infty$. 

\end{theorem}

\begin{proof}
Starting from \eqref{eq:5} we write, on the one hand,  
$$a_n  \frac{f(a_nx)}{F(a_nx)} = \frac{\alpha x^{-\alpha-1}}{n} \left(
  1+O(r_n) \right) 
$$
and, on the other hand, 
$$a_n \frac{f'(a_nx)}{f(a_nx)} =
-\frac{\alpha+1}{x}+ \alpha \frac{x^{-\alpha-1}}{n}\left(
  1+O(r_n) \right)$$
so that 
\begin{equation*}
  \rho_n(W_n)  = \alpha W_n^{-\alpha-1}\left( 1+O(r_n) \right) - \frac{\alpha+1}{W_n}.
\end{equation*}
Consequently  
\begin{align*}
  1 -  W_n^{\alpha} \left(
        \frac{\alpha+1}{\alpha} +  \frac{W_n}{\alpha}\rho_n(W_n)
      \right)  & =  - O(r_n).
\end{align*}
Plugging this expression into \eqref{eq:31} we get \eqref{eq:13}.
\end{proof}

Let $\mathcal{H}$ be any class of functions and define the probability
distance $d_{\mathcal{H}}(X, Y) =
\sup_{h \in \mathcal{H}} \left| E h(X) - E h(Y) \right|$.   Taking
suprema on each side of \eqref{eq:31} we get  
 \begin{equation}\label{eq:10} 
    d_{\mathcal{H}}(W_n, G) = \kappa_{\mathcal{H}} O(r_n)
  \end{equation}
for $\kappa_{\mathcal{H}} =
\alpha \sup_{h \in \mathcal{H}} E \left|
  \mathcal{T}_{\alpha}^{-1}h\right|$. 
The following lemma then settles the question of convergence in
Kolmogorov distance. 
\begin{lemma}\label{lem:indic}
If $h(x) = \mathbb{I}(x \le t)$ and $\varphi_h =
\mathcal{T}_{\alpha}^{-1}h$ then 
$ \| \varphi_{h}\|_{\infty}   \le 1/ \alpha$.
\end{lemma}
\begin{proof}
With $h(x) = \mathbb{I}(x\le t)$ we apply \eqref{eq:solstfre} to get 
   \begin{equation*}
\alpha \varphi_h(x) = \begin{cases}   1-e^{-t^{-\alpha}} &\mbox{if } x \leq t, \\ 
 e^{-t^{-\alpha}} (e^{x^{-\alpha}}-1) & \mbox{if } x>t. \end{cases} 
\end{equation*}
We observe that $ e^{-t^{-\alpha}} (e^{x^{-\alpha}}-1) <   (1-e^{-t^{-\alpha}})$
if $x>t$ so that 
\begin{align*}
|  \varphi_h(x) | \le (1-e^{-t^{-\alpha}})/\alpha 
\end{align*}
for all $x\ge0$.  The function $1-e^{-t^{-\alpha}}$ is strictly decreasing
for $t \in (0, \infty)$. Thus, 
$\| \varphi_{h}\|_{\infty} \le 1/\alpha$ for all $t$.  
 \end{proof}
\begin{corollary}\label{cor:smith}
As $n \to \infty$ we have
  \begin{equation*}
    \sup_x \left| F^n(a_nx) - \Phi_{\alpha}(x) \right| = O(r_n).
  \end{equation*}
\end{corollary}

A local limit version of Corollary \ref{cor:smith} (uniform convergence of the respective
pdfs) was first obtained
in  \cite{omey1988rates}. In our notations it suffices
to take $h_u(x) = \delta(x=u)$  (a Dirac delta)
in  (\ref{eq:13}) to get
\begin{align*}
  f_n(u) -\phi_\alpha(u) = -\alpha E \left[\varphi_u(W_n)  \right] O(r_n).
\end{align*}
The problem remains to understand  the quantity
\begin{equation}
  \label{eq:8}
 E \left[  \varphi_u(W_n)\right] =\frac{\phi_\alpha(u) }{\alpha}\left( \int_u^{\infty} e^{w^{-\alpha}}
    f_n(w)dw-1 \right)
\end{equation}
in terms of  $u \in \R$. 

\begin{example} Consider, as in Example \ref{ex:pareto}, the maximum
  of Pareto random variables (suitably normalized). Then it is easy to
  see that 
  \begin{equation*}
     \left| E \left[  \varphi_u(W_n)\right]  \right| \le \frac{1}{\alpha}\left| u^{-\alpha}-1 \right| \phi_{\alpha}(u)
  \end{equation*}
for all $\alpha$ and all $u>1$. The rhs is bounded uniformly in $u$ by
a function of $\alpha$ (which is easy to write out explicitly); in
particular it is bounded by 1 for all $\alpha$.  
\end{example}
Still under the assumption \eqref{eq:3} we can apply the same
tools as in the proof of Theorem~\ref{theo:1} to deduce 
\begin{equation*}
  E \left[  \varphi_u(W_n)\right] =  \phi_\alpha(u) \left(
    \frac{1-(O(r_n)+1)e^{-O(r_n)u^{-\alpha}}}{\alpha O(r_n)} \right). 
\end{equation*}
This quantity, as a function of $u$, is bounded uniformly by some
constant depending on $\alpha$. 
\begin{corollary}
  Still under Assumption~A we have 
  \begin{equation*}
    \sup_{u \in \R}\left| f_n(u) - \phi_{\alpha}(u)   \right| \le
    O(r_n) \kappa_{\alpha}
  \end{equation*}
with $\kappa_{\alpha} = \sup_{u >0} \sup_n \left|  E \left[  \varphi_u(W_n)\right]  \right| $.
\end{corollary}

\section{Acknowledgements}
\label{sec:acknowledgements}

The research of Carine Bartholm\'e is supported by a F.R.I.A. grant
from the Fonds  de la Recherche Scientifique - FNRS (Belgium). 

\bibliographystyle{model2-names.bst}
\bibliography{../../../Bibliography/biblio_ys_stein.bib}

\end{document}